\documentclass[11pt,reqno,letterpaper]{amsart}
\usepackage{mbcat}
\usepackage{fullpage}
\title{Functorial compactification of linear spaces}
\author{Chris Kottke}
\address{New College of Florida}
\email{ckottke@ncf.edu}
\begin{document}
\begin{abstract}
We define compactifications of vector spaces which are functorial
with respect to certain linear maps. These ``many-body'' compactifications are 
manifolds with corners, and the linear maps lift to b-maps in the sense of Melrose.
We derive a simple criterion under which the lifted maps are in fact b-fibrations, and identify
how these restrict to boundary hypersurfaces.
This theory is an application of a general result on the iterated blow-up
of cleanly intersecting submanifolds which extends related results in the literature.
\end{abstract}
\maketitle

\newcommand\Lin{\mathsf{Lin}}
\newcommand\MwC{\mathsf{MwC}}

\section{Introduction} \label{S:intro}

One approach to the inherent challenge of solving an analytical problem on a
non-compact space is to look for the ``right compactification'' and seek a
solution there.
To facilitate the analysis, the compactification should reside in a good
category, say a smooth manifold with boundary or with corners, and the new
points in the compactification should have some geometrical or analytical
significance.
For example, the (possibly singular) limits of a solution at the boundary faces
of a compact manifold with corners typically encode important asymptotic
regimes of original problem on the non-compact interior.
While examples of this approach are far too numerous to list completely, we
highlight
\cite{mazzeo1987meromorphic,mazzeo1990adiabatic,melrose1993atiyah,melrose1994spectral,graham2003scattering,vasy2010wave}
as a small sample of particular historical importance.
In general, the question of which compactification is the ``right'' one is
highly problem-specific and not always immediately obvious, even when the original space
is something simple, like a vector space.

In this paper we present a class of compactifications of vector spaces with the
key property that they admit smooth extensions of a given set of linear maps,
or in other words are {\em functorial}.
The starting point is a category, $\Lin$, whose objects are finite dimensional
real vector spaces $X$ equipped with a {\em linear system}---finite sets,
$\cS_X$, of subspaces including $\set 0$ and $X$ which are closed under
intersection---and whose morphisms are the {\em admissible linear maps} $f : X
\to Y$, which are required to satisfy $f^\inv(V) \in \cS_X$ for every $V \in
\cS_Y$.
The {\em many body compactification}, $\wh X$, of a vector space $X \in \Lin$
is the manifold with corners obtained from the radial compactification, $\ol X$,
by the iterated blow-up of each boundary $\pa \ol V$ of a subspace
$V \in \cS_X$, in order from smallest to largest.
This compactification is known in the literature going back to the work of Vasy
on scattering theory in many body systems \cite{vasy2001propagation}; however,
that the association
$
	X \mapsto \wh X
$
is actually a {\em functor} from the category $\Lin$ to the category of
manifolds with corners and b-maps is a new observation leading
to novel applications which are discussed below.

In fact, more can be said: if an admissible map $f : X \to Y$ is a so-called
{\em admissible quotient}, then its functorially associated morphism $\wh f$ is
actually a {\em b-fibration}, the analogue for corners of a fibration in the
usual category of manifolds without corners.
In this case we obtain a detailed description of the restriction of $\wh f$ to
the boundary hypersurfaces of $\wh X$.
Referring to \S\ref{S:bkg} precise definitions, our results may be summarized
as follows.
\begin{thm-}[Theorem~\ref{T:mb_functor}, \ref{T:boundary_faces}, and \ref{T:mb_map_on_hs}]
\mbox{}
\begin{enumerate}
[{\normalfont (i)}]
\item\label{I:functor}
The association $X \mapsto \wh X$ is a functor from $\Lin$ to the category
of manifolds with corners, so every admissible
map $f : X \to Y$ extends to a unique b-map $\wh f : \wh X \to \wh Y$.
If $f$ is an admissible quotient, meaning $f(\cS_X) =
\cS_Y$, then $\wh f$ is a b-fibration.
\item \label{I:faces}
The boundary hypersurfaces of $\wh X$, which are indexed by $V \in \cS_X
\setminus \set 0$, are diffeomorphic to products $B_V \times \wh{X/V}$, where
$\wh{X/V}$ is the many body compactification of the quotient $X/V$, and
$B_V$ denotes the space obtained from the sphere $\pa \ol V$ by the blow-up of each subsphere $\pa
\ol W$ where $W \in \cS_X$ and $W \subset V$.
\item\label{I:f_on_hs}
In the case that $f : X \to Y$ is an admissible quotient, the restriction of
the b-fibration $\wh f$ to each boundary hypersurface $B_V \times \wh{X/V}$ of
$\wh X$ is identified with the product $\pa \wh f_V \times \wh{f/f_V}$, where
$\pa \wh f_V : B_V \to B_{f(V)}$ is induced by the extension of the map $f
\rst_V : V \to f(V)$ and $\wh{f/f_V} : \wh{X/V} \to \wh{Y/f(V)}$ is the
extension of the quotient map $f/f_V : X/V \to Y/f(V)$.
\end{enumerate}
\end{thm-}

This theory leads to a particularly simple solution to a problem posed in
\cite{melrose2008scattering} by Melrose and Singer, namely, to find
compactifications of the products $X^n$, $n \in \bbN$ which lift all of the
projections $X^n \to X^m$ for $m < n$ as well as the difference maps $X^n \to
X$, $(x_1,\ldots,x_n) \mapsto x_i - x_j$. (See \S\ref{S:appl_scatprod}.)

In addition, the many body compactifications of vector spaces
serve as key examples of a structure of interest in the manifolds with corners
community known as a {\em fibered corners} structure
\cite{albin2010resolution,albin2012signature,debord2015pseudodifferential,conlon2016quasi}
(see Remark~\ref{Rmk:qfb} for details).
Further equipping the vector spaces $X$ with Euclidean metrics, the
compactifications $\wh X$ become examples of {\em quasi fibered boundary} (QFB)
manifolds, a class of spaces which has been of recent
interest in Calabi-Yau and hyperK\"ahler geometry \cite{conlon2016quasi},
as it generalizes the QALE manifolds of Joyce \cite{joyce2000compact}.
While of course not topologically interesting, the $\wh X$ have the advantage
of a wealth of easily constructed b-maps (and in particular b-fibrations)
coming from the underlying linear maps, making them important test cases for
developing analytical results on QFB spaces.

In particular, they will play an important role in a forthcoming work of the
author along with K.\ Fritzsch and M.\ Singer on a QFB compactification of the
hyperK\"ahler moduli space of $\SU(2)$ monopoles, where in addition to serving
as simplified models for the compactified moduli space, the functoriality of
many body spaces is an essential part of the analytical machinery
used to obtain the compactification.

\medskip

The functoriality of the many body compactifications is a consequence
of our other main result (of independent interest) concerning
the iterated blow-up of submanifolds in a manifold with corners.
If $\cP$ is a finite set of
p-submanifolds of a manifold with corners $M$ (meaning submanifolds positioned nicely with respect to the
boundary faces of $M$; see \S\ref{S:bkg} for a precise definition),
we say $\cP$ is {\em closed under clean intersection} if each pair $P_i, P_j
\in \cP$ intersects cleanly and its intersection $P_i \cap P_j$ is also an
element of $\cP$. 
Among the many possible total orders on $\cP$,
a {\em size order} is any one which extends the partial order by
inclusion, so that $P_i \cap P_j$ must precede both $P_i$ and $P_j$.
More generally, an {\em intersection order} is any total order in which either
$P_i$ or $P_j$ is allowed to precede $P_i \cap P_j$, but not both.

\begin{thm-}[Theorem~\ref{T:comm_blowup_order}, Corollary~\ref{C:comm_blowup_int_order}]
\mbox{}
\begin{enumerate}
[{\quad\normalfont (i)}]
\item 
If $\cP$ is closed under clean intersection, then the iterated blow-up 
\[
	[M; \cP] := [M; P_1,\ldots, P_N]
\]
is well-defined up to diffeomorphism, where $P_1,\ldots, P_N$ is any
intersection order of the elements in $\cP$; in particular the iterated
blow-ups taken in two different intersection orders are canonically
diffeomorphic.
\item
If $\cQ \subset \cP$ is an intersection closed subset, then the blow-down $[M; \cP] \to M$
factors through a unique b-map $[M; \cP] \to [M; \cQ]$.
\end{enumerate}
\end{thm-}

Part (i) generalizes a result proved in \cite{vasy2001propagation} for size orders, under an additional assumption that 
each pair $\set{P_i,P_j}$ forms a so-called normal family, while
parts (i) and (ii) were proved in \cite{melrose2008scattering} in the special case that $\cP$ is a set of boundary faces.

\medskip An outline of the paper is as follows.
In \S\ref{S:general}, we define the category of linear systems and admissible
maps.
\S\ref{S:bkg} contains background information on manifolds with corners and
blow-up, and concludes with a proof of the second Theorem.
In \S\ref{S:functor} we prove the functoriality of the many body
compactification---part \eqref{I:functor} of the first Theorem
above---and in \S\ref{S:boundary} we characterize the boundary hypersurfaces of
the many body compactification, proving parts \eqref{I:faces} and
\eqref{I:f_on_hs}.
Finally, as an application, we show in \S\ref{S:appl_scatprod} how appropriate
many body compactifications of the products $X^n$ furnish an alternative to the
{\em scattering products} of \cite{melrose2008scattering}.

\subsection*{Acknowledgments} The author would like to thank Daniel Grieser for
his detailed comments on the manuscript, Michael Singer and Richard Melrose for helpful
discussions and comments on early versions of this paper, and Pat McDonald for his suggestions
regarding the exposition.

\section{Linear systems and many body compactification} \label{S:general}

\begin{defn}
A {\em linear system} in a vector space $X$ is a finite set $\cS_X$ of subspaces of $X$ such that
\begin{enumerate}
[{\normalfont (i)}]
\item $\set 0$ and $X$ are in $\cS_X$, and
\item whenever $W$ and $V$ are in $\cS_X$, then $W\cap V$ is in $\cS_X$.
\end{enumerate}
Given another vector space $Y$ with linear system $\cS_Y$, we say a linear map
$f : X \to Y$ is {\em admissible} if $f^{-1}(\cS_Y) \subset \cS_X$, that is, if
$f^{-1}(W) \in \cS_X$ for every $W \in \cS_Y$.
In particular, $\ker f$ is required to be in $\cS_X$.
The collection of finite dimensional vector spaces with linear systems and
admissible maps form a category $\Lin$.
\label{D:lin}
\end{defn}

If $K$ is a subspace of $X$ (not necessarily in $\cS_X$), then $\cS_X/K :=
\set{V/(V \cap K) : V \in \cS_X}$ is a linear system in the quotient $X/K$.
In general, the quotient map $\pi : X \to X/K$ need not be admissible.
In fact, admissibility of $\pi$ is equivalent to the condition that $V + K$ is
in $\cS_X$ for every $V$ in $\cS_X$ (in particular, for $V = \set 0$).
By an {\em admissible quotient map}, we will understand that the target
$X/K$ is equipped with the linear system $\cS_X/K$ (as opposed to some
subsystem, with respect to which the quotient would still be admissible).

\begin{lem}
The following are equivalent:
\begin{enumerate}
[{\normalfont (i)}]
\item $f : X \to Y$ is admissible and $f(\cS_X) := \set{f(V) : V \in \cS_X} = \cS_Y$.
\item $Y \cong X/\ker f$, and $f$ is an admissible quotient map.
\end{enumerate}
\label{L:ad_quot}
\end{lem}
\begin{proof}[Proof of Lemma~\ref{L:ad_quot}]
The condition $f(\cS_X) = \cS_Y$ implies in particular that $Y = f(X)$, so $f$
is surjective and we may identify $Y$ with $X/K$, where $K = \ker f$.
Under this identification, $f(V) \cong V/(V \cap K)$,
so the condition $f(\cS_X) = \cS_Y$ becomes the statement that $\cS_Y \cong
\cS_X / K$.
\end{proof}

\begin{defn}
The {\em many body compactification} of a linear system $(X,\cS_X)$ is the manifold with corners
\[
	\wh X = [\ol X; \pa \cS_X],
	\quad \pa \cS_X = \set{\pa \ol V : V \in \cS_X}
\]
obtained by iteratively blowing up the boundaries of the subspaces inside the
radial compactification $\ol X$ of $X$.
The blow-up is performed in some {\em size order}, meaning any total order on
$\cS_X$ extending the partial order defined by inclusion of subspaces.
We prove in Corollary~\ref{C:whX_welldefined} that $\wh X$ is well-defined,
following a digression through the theory of manifolds with corners,
referring the reader to \cite{melrose1993atiyah,melrosedifferential} for a
comprehensive account.
\label{D:mbcpctn}
\end{defn}

\section{On manifolds with corners} \label{S:bkg}
Recall that the {\em radial compactification} of a vector space $X$ is the
manifold with boundary $\ol X$ obtained by adjoining a sphere of dimension
$\dim(X) - 1$ at infinity, with the smooth structure induced by taking $1/r$ as
a boundary defining function for any choice of Euclidean norm $r$.
Equivalently, $\ol X$ can be obtained by the stereographic projection of $X
\times \set{1} \subset X \times [0,\infty)$ onto the half sphere $\sset{(x,t)
\in X \times [0,\infty) : \abs x^2 + t^2 = 1}$.

As a manifold with boundary, $\ol X$ is also a {\em manifold with corners},
meaning a second countable Hausdorff space $M$ with a covering by charts
locally diffeomorphic to open sets in $\bbR^n_+ = [0,\infty)^n$.
The {\em codimension} of a point in $M$ is the number of
$\bbR_+$ coordinates vanishing at that point in any chart, and a {\em 
face} of $M$ is the closure of a connected component of points with fixed
codimension;
we use the term {\em boundary face} if the codimension is positive.
In particular, {\em boundary hypersurfaces} are boundary faces of codimension
one, and we require as part of the definition of a manifold with corners that
these are embedded---that is, upon taking the closure of a set of
points of codimension one, no self-intersections occur.

A {\em p-submanifold} $P \subset M$ is a submanifold of some face $F$
of $M$, with $F$ taken as small as possible, which is required to
intersect any boundary hypersurface of $M$ transversally in $F$; in particular
$P$ is covered by ``product-type'' coordinate charts valued in $\bbR_+^{n-l}
\times \bbR^l$ for various $l$, in which it is locally defined by the vanishing
of $\codim(P)$ of the coordinates.
It follows from the definition that if $P$ and $Q$ are p-submanifolds of $M$
and $P \subseteq Q$, then $P$ is also a p-submanifold of $Q$ (there is some
confusion about this point in the literature).
The (radial) {\em blow-up} of $P$ in $M$, denoted by $[M; P]$, is the space $(M
\setminus P) \cup S_+P$ where $S_+ P$ is the inward pointing spherical normal
bundle (that is, the set of unit vectors in $NP$ with respect to some inner
product which integrate to flows of $M$) with smooth structure induced by polar
coordinates normal to $P$.
This space admits a {\em blow-down map} $\beta : [M; P] \to M$, which is the
smooth surjection defined by the identity on $M\setminus P$ and the bundle
projection on $S_+P$; this is an example of a {\em b-map}, as defined in
\S\ref{S:functor} below.
The boundary hypersurface $\beta^\inv(P) \subset [M; P]$ is referred to as the
{\em front face} of the blow-up.
The following local coordinate characterization is convenient: given a
coordinate chart in which $P$ has the form $\set{(x_1,\ldots, x_n) : x_1 =
\cdots = x_k = 0}$ where the $x_i$ are each valued in either $\bbR_+$ or
$\bbR$, the preimage in $[M; P]$ is covered by charts with ``projective
coordinates''
\begin{equation}
	(\wh x, x'/\wh x),
	\quad \wh x = \pm x_i \geq 0, \ i \in \set{1,\ldots,k},
	\quad x' = (x_1,\ldots,x_{i-1},x_{i+1},\ldots, x_n)
	\label{E:proj_coord_notn}
\end{equation}
where $\wh x$ is a $\bbR_+$ coordinate ranging over
$\pm x_1,\ldots, \pm x_k$ (we take only the $+$ sign if the original
coordinate was $\bbR_+$-valued, and both signs if it was
$\bbR$-valued), and $x'$ denotes the tuple $(x_1,\ldots, x_n)$ with $x_i$
removed; in such a chart $\wh x$ is a boundary defining coordinate for the
front face.

If $S \subset M$ is another p-submanifold, the {\em lift} (or {\em proper
transform}) of $S$ in $[M; P]$ is defined to be the set $\beta^\inv(S)$ if $S
\subset P$, and the closure of $\beta^\inv(S\setminus P)$ in $[M; P]$
otherwise.
Committing a minor abuse of notation, we will continue to use the same letter to
denote both a p-submanifold of $M$ and its lift to $[M; P]$.
Finally, we use the abbreviated notation $[M; P, S]$ to denote the {\em
iterated blow-up} $[[M; P]; S]$, given by first blowing up $P$ in $M$ and then
the lift of $S$ in $[M; P]$, provided that this lift is a p-submanifold in $[M;P]$.

In the setting of Definition~\ref{D:mbcpctn} above, we may identify the closure
of each $V \in \cS_X$ in $\ol X$ with its radial compactification $\ol V$;
these along with their boundaries $\pa \ol V$ are clearly p-submanifolds in $\ol X$.
That $\wh X$ is well-defined is a direct consequence of the following general
results about commuting blow-ups in manifolds with corners.

Recall that
a set
$\cP = \set{P_1,\ldots, P_\ell}$
of p-submanifolds
of $M$ is
a {\em normal family} if it admits {\em simultaneous} product-type
coordinates; in other words, $\bigcup_i P_i$ is covered by coordinate charts
$\set{(x_1,\ldots,x_n)}$ in which every $P_j$ has the form
$\set{(x_1,\ldots,x_n) : x_i = 0, i \in I_j}$ for some $I_j \subset
\set{1,\ldots,n}$.

\begin{prop}[\cite{melrosedifferential}, Prop.~5.8.1 and Prop.~5.8.2]
Let $M$ be a manifold with corners, with p-submanifolds $P, Q \subset M$. If one of the three following conditions hold:
\begin{enumerate}
[{\normalfont (i)}]
\item \label{I:comm_blowup_contained}
one of the submanifolds is included in the other, say $P \subset Q$, or
\item \label{I:comm_blowup_disjoint}
$P$ and $Q$ are disjoint in $M$, or
\item \label{I:comm_blowup_transverse}
$P$ and $Q$ are {\em normally transverse}, meaning $P$ and $Q$ constitute a normal family and intersect transversally in $M$,
\end{enumerate}
then there is a natural diffeomorphism $[M; P,Q] \cong [M; Q,P]$.
\label{P:comm_blowup}
\end{prop}
Though this result is well-known, we include a proof below since the arguments
we employ form the basis for other results later on.
%



\begin{proof}[Proof of Proposition~\ref{P:comm_blowup}]
Observe that in all three cases $\set{P,Q}$ is a normal family.
Indeed, by induction any nested sequence $P_1 \subseteq P_2 \subseteq \cdots \subseteq P_\ell$
of p-submanifolds forms a normal family, since we
may assume that we have simultaneous product coordinates for $\set{P_1,\ldots,
P_k}$ and use the fact that $P_k$ is a p-submanifold of $P_{k+1}$ to extend
these to simultaneous local product coordinates for $\set{P_1,\ldots,P_{k+1}}$.

If $P$ and $Q$ are disjoint, then $[M; P,Q] \cong [M; Q,P]$ obviously holds.
Suppose next that $P \subset Q$.
%
This means that
locally along $P$ there exist coordinates $(x,y,z)$, where $P = \set{y = 0,
z = 0}$ and $Q = \set{z = 0}$.
Here $x$, $y$ and $z$ are tuples (e.g. $z = (z_1,\ldots,z_k)$) each component
of which is valued either
in $\bbR$ or $\bbR_+$.
The preimage of such a coordinate chart in $[M; Q]$ is covered by charts having
coordinates of the form $(x,y, \wh z, z'/\wh z)$, where in notation as above $\wh z \in \bbR_+$ runs
through $\pm z_i$, $i = 1,\ldots, k$, and $z'$ denotes the tuple $z$ with
$z_i$ removed.
%
%
The lift of $P$ in any such chart is the set $\set{y = 0, \wh z = 0}$.
On the blow-up $[M; Q, P]$, we then have coordinates $(x,y/\wh z, \wh z, z'/\wh z)$
and $(x,\wh y, y'/\wh y, \wh z/\wh y, z'/\wh z)$, using similar notation.

In the other direction, $[M; P]$ has coordinates of the form $(x,\wh y, y'/ \wh
y, z/\wh y)$, in which $Q$ lifts as $\set{z/\wh y = 0}$, and of the form
$(x,y/\wh z, \wh z, z'/\wh z)$, which $Q$ does not meet.
Further passing to $[M; P; Q]$, we have a cover by coordinates of the form
$(x,\wh y, y'/\wh y, \wh z/\wh y, (z'/\wh y)/(\wh z/ \wh y) \equiv z'/\wh z)$.
Clearly we can identify these coordinate charts in $[M; P, Q]$ with the
respective ones from $[M; Q, P]$, which patch together into a diffeomorphism.

Finally suppose $P$ and $Q$ are normally transverse; this amounts to the local existence
of coordinates $(x,y,z)$ as above in which now $P$ is the set $\set{z = 0}$ and $Q$ is the
set $\set{y = 0}$.
The blow-up $[M; P]$ admits coordinate charts of the form $(x,y, \wh z, z'/\wh z)$ with $Q$
lifting as $\set{y = 0}$, so that $[M; P, Q]$ has charts of the form $(x,\wh y, y'/\wh y, \wh z, z'/\wh z)$.
This is clearly symmetric upon interchanging the roles of $P$ and $Q$, so we can again
patch together a diffeomorphism $[M; P, Q] \cong [M; Q, P]$.
\end{proof}

\begin{thm}
Let $M$ be a manifold with corners, and $\cP$ a collection of p-submanifolds
which is closed under clean intersection (meaning the submanifolds intersect
cleanly pairwise, and each such non-empty intersection is an element of $\cP$).
%
%
\begin{enumerate}
[{\normalfont (i)}]
\item \label{I:comm_blowup_order_defined}
The iterative blow-up $[M; \cP] = [\cdots[[M; P_1]; P_2]; \cdots; P_N]$ is
well-defined where $P_1 < P_2 < \cdots <P_N$ is any choice of {\em size order},
meaning a total order on $\cP$ extending the partial order by inclusion.
\item \label{I:comm_blowup_order_factor}
If $\cQ \subset \cP$ is an intersection closed subset of $\cP$, then the blow-down $[M; \cP] \to M$ factors
through a unique b-map $[M; \cP] \to [M; \cQ]$.
\end{enumerate}
\label{T:comm_blowup_order}
\end{thm}
\begin{proof}
In part \eqref{I:comm_blowup_order_defined}, given a choice of total order on
$\cP$, we must first justify why the iterated blow-up is defined; more
precisely we must show that upon blowing up some $Q$, those $P$ such that $Q <
P$ lift again to p-submanifolds.
If $Q$ and $P$ are disjoint in $M$ to begin with then this is obvious.
If $Q$ and $P$ meet but $Q \not \subset P$, then by the clean intersection
property $Q$ and $P$ lift to disjoint p-submanifolds in the blow-up of $Q \cap
P$, which must precede both in the given order.
Finally, if $Q \subset P$ in $M$, then this inclusion relation persists to their lifts
under blow-up of those elements preceding $Q$ in the size order, and then the
fact that $P$ lifts to a p-submanifold upon blowing up $Q$ was observed in the
local coordinate computation in the proof of Proposition~\ref{P:comm_blowup}
above.

To see that $[M; \cP]$ is well-defined independent of the choice of
size order, note that any size order may be obtained from any other one by a
sequence of size orders in which pairs of adjacent elements (necessarily
incomparable in the original partial order) are swapped.
Thus we consider a pair $P$ and $Q$ of submanifolds such that $P \not \subset
Q$ and $Q \not \subset P$.
%
%
In the blow-up of $P\cap Q$ (if non-empty) in $M$ the lifts of $P$ and $Q$ are
made disjoint (or else they are already disjoint in $M$) so it follows from
Proposition~\ref{P:comm_blowup} that
\begin{equation}
	[M; \ldots, P\cap Q, \ldots, P, Q,\ldots] \cong [M; \ldots, P\cap Q, \ldots, Q,P,\ldots].
	\label{E:comm_blowup}
\end{equation}

For part \eqref{I:comm_blowup_order_factor}, it suffices to suppose that $\cP$ is larger than $\cQ$ by
precisely one element, say $\cP = \cQ \cup \set {P}$, and to suppose that the
chosen size order on $\cP$ restricts to the size order on $\cQ$.
Then we must show that $[M; \cQ, P] := [[M; \cQ]; P]$ (which obviously factors
through $[M; \cQ]$) is well-defined and is diffeomorphic to $[M; \cP]$; this involves
showing that $P$ lifts to a p-submanifold in $[M; \cQ]$ and that we may commute
$P$ past all those elements $Q$ of $\cQ$ which follow it in the size order on $\cP$.

%
If $P$ and $Q$ are disjoint, then $P$ lifts trivially under the blow-up of $Q$,
and we may commute $Q$ and $P$ in any blow-up order.
If $P$ meets $Q$ but is not contained in $Q$, then the lifts of $P$ and $Q$ are
disjoint p-submanifolds in the blow-up of $P\cap Q$, which must be in $\cQ$ by
intersection closure, and in this case \eqref{E:comm_blowup} holds.
Finally, if $P \subset Q$ in $M$ then one of two possibilities occurs.
If none of the elements preceding $Q$ contains $P$, then $P$ lifts to a
p-submanifold in the blow-up of these preceding elements by the arguments just
given; moreover this lift of $P$ is contained in the lift of $Q$, so $[M;
\ldots, Q,P, \ldots] \cong [M; \ldots, P,Q, \ldots]$ by part
\eqref{I:comm_blowup_contained} of Proposition~\ref{P:comm_blowup}.

On the other hand, if there is some $R$ preceding $Q$ such that $P \subset R$,
then the lifts of $P$ and $Q$ are not comparable upon blowing up $R$.
Taking a minimal such $R$, we may assume that $P \subset R \subset Q$ and that
no element preceding $R$ contains $P$ (so that the relation $P \subset R \subset Q$ continues
to hold after blowing up the elements preceding $R$).
Then we claim that $P$ and $Q$ lift to normally transverse p-submanifolds in the blow-up
of $R$.
Indeed, over local coordinates $(w,x,y,z)$ in which $P = \set{x = 0, y = 0, z =
0}$, $R = \set{y = 0, z = 0}$ and $Q = \set{z = 0}$, the blow-up of $R$ is
covered by coordinates of the form $(w,x,\wh y, y'/\wh y, z/\wh y)$ and
$(w,x,y/\wh z, \wh z, z'/\wh z)$.
In the latter coordinates $P$ lifts to $\set{x = 0, \wh z = 0}$, while the lift
of $Q$ is empty, and in the former coordinates $P$ lifts to $\set{x = 0, \wh y
= 0}$ while $Q$ lifts to $\set{z/\wh y = 0}$, which are normally transverse.
Thus in this final case $P$ lifts to a p-submanifold in the blow-up of $Q$, and
$[M; \ldots, R, \ldots, Q,P, \ldots] \cong [M; \ldots, R, \ldots, P,Q,
\ldots]$, completing the proof.
\end{proof}

\begin{rmk}
In \cite{vasy2001propagation}, (c.f.\ Lemmas~2.7 and 2.8), Vasy proves a result
similar to part (i) under the additional hypothesis that the p-submanifolds are
pairwise normal, i.e., each pair $\set{P_i,P_j}$ is a normal family.
In the proof above, we only use normality for nested sequences (which is automatic),
so this additional hypothesis can be removed.
\end{rmk}

\begin{cor}
The many body compactification $\wh X$ of a linear system $(X, \cS_X)$ is well-defined.
\label{C:whX_welldefined}
\end{cor}

In fact it is possible to strengthen Theorem~\ref{T:comm_blowup_order} which we
do here for the sake of completeness, though we shall not need the stronger
version below.
Given a finite collection $\cP$ of p-submanifolds closed under clean
intersection, define an {\em intersection order} on $\cP$ to be any total order
in which for any $P,Q \in \cP$, the intersection $P \cap Q$ is not preceded by
both $P$ and $Q$.
In other words, for each distinct triple $P$, $Q$ and $P \cap Q$ in $\cP$, one
of the following four possibilities must hold:
(i) $P \cap Q < P < Q$, (ii) $P \cap Q < Q < P$, (iii) $Q < P \cap Q < P$, or (iv) $P < P \cap Q < Q$.

It is not obvious a priori that the blow-up in $M$ of $\cP$ in some
intersection order is even well-defined, since upon blowing up some $Q$, in
general there will be $P > Q$ whose lifts are not initially p-submanifolds.
The simplest example of this situation occurs in the half space $M = \bbR^2\times \bbR_+$,
with $Q = \set{x = y = 0}$ the $z$-axis and $P = \set{x = z,\ y = 0}$ a ``diagonal'' line
meeting $Q$ cleanly at the boundary point $P \cap Q = (0,0,0)$.
Nevertheless, it so happens that the lift of such $P$ eventually become
p-submanifolds after further blow-ups.

\begin{cor}
Theorem~\ref{T:comm_blowup_order} holds with ``size order'' replaced by ``intersection order''.
\label{C:comm_blowup_int_order}
\end{cor}

\begin{proof}
It suffices to show that the iterated blow-up of $\cP$ in an intersection order
is well-defined and diffeomorphic to the blow-up of $\cP$ in some size order.
We do this by induction on the size of $\cP$, the case $\abs{\cP} = 1$ being
trivial.
(Also, the case $\abs{\cP} = 2$ is covered by Proposition~\ref{P:comm_blowup}.)
Thus, given an intersection ordered set $\cP = \set{P_1 < \cdots < P_n}$, we
assume by induction that $[M; P_1,\ldots,P_{n-1}]$ is well-defined, and without
any loss of generality we may assume that $\cQ = \set{P_1 < \cdots < P_{n-1}}$
is in a size order.
Then that $[M; P_1,\ldots,P_{n-1},P_n] = [M; \cQ, P_n]$ is well-defined and
diffeomorphic to a blow-up of $\cP$ in $M$ in a size order was shown in the proof
of part (ii) of Theorem~\ref{T:comm_blowup_order} above.
\end{proof}
\begin{rmk}
This generalizes Proposition~3.5 and Corollary~3.8 in
\cite{melrose2008scattering}, where the authors prove
Corollary~\ref{C:comm_blowup_int_order} in the case that $\cP$ and $\cQ$ are
collections of boundary faces of $M$.
In fact it is not quite a full generalization of their results, since Melrose
and Singer relax the condition that the sets be closed under intersection,
requiring only that they be closed under {\em non-transversal intersection},
meaning that $P \cap Q$ is only required to be in $\cP$ if $P$ and $Q$ are not
(normally) transverse.
(Note that transverse boundary faces are automatically normally transverse.)
The proofs of Theorems~\ref{T:comm_blowup_order} and
Corollary~\ref{C:comm_blowup_int_order} would go through under this weaker
hypothesis thanks to
Proposition~\ref{P:comm_blowup}.\eqref{I:comm_blowup_transverse}, {\em
provided} the lifts of a normally transverse pair $P$ and $Q$ remain normally
transverse upon blowing up the elements preceding them in any intersection
order.
While this is automatic for boundary faces, it is not clear (to the author at least) that
it holds for pairwise cleanly intersecting p-submanifolds without further
hypotheses.

The intersection order condition in
Corollary~\ref{C:comm_blowup_int_order} is sharp in the sense that $[M; P,Q,
P\cap Q]$, even if well-defined, is generally not diffeomorphic to $[M; P \cap
Q, P, Q]$.
This is evident in simple examples, such as that of two distinct lines meeting
at a point in $\bbR^3$.
\end{rmk}

\section{Many body compactification as a functor} \label{S:functor}

Returning to our original setting, we now show that the many body
compactification is a functor from $\Lin$ to the category $\MwC$ of manifolds
with corners.
While there are various choices of morphisms between manifolds with corners (in
addition to the conventions of \cite{melrose1993atiyah} observed here, compare
for instance \cite{joyce2012manifolds} or \cite{joyce2016manifolds}, Definition
2.1), we take the morphisms in $\MwC$ to be the {\em b-maps}\footnote{These are
called ``smooth maps'' in \cite{joyce2016manifolds}. Joyce uses the term
``weakly smooth'' for what we call smooth here.} $g : M \to N$, which are by
definition those smooth maps (i.e., $g^\ast\big(C^\infty(N)\big) \subset
C^\infty(M)$) such that for each boundary defining function $\rho_{H}$ of a
boundary hypersurface $H \subset N$, the pullback $g^\ast(\rho_H)$ either
vanishes identically (implying that $g(M) \subset H$) or has the form
\begin{equation}
	g^\ast(\rho_H) = a\prod_{H'} \rho_{H'}^{e_{HH'}},
	\quad e_{HH'} \in \bbN_0,
	\quad a > 0 \in C^\infty(M).
	\label{E:b_map}
\end{equation}
Here the index $H'$ ranges over boundary hypersurfaces of $M$, and $\rho_{H'}$
is a boundary defining function for $H'$.
In this note all b-maps are {\em interior}, meaning that \eqref{E:b_map} always
holds.
Examples include the blow-down maps $\beta : [M; P] \to M$.
Of particular importance are the {\em b-fibrations}, which are fibrations in
the usual sense over the interiors and restrict again to b-fibrations over each
boundary face of the domain to some boundary face of the range.
They are defined to be b-maps whose natural differential (c.f. \cite{melrose1993atiyah}) is
surjective pointwise, and for which at most one exponent $e_{HH'}$ is nonzero
for each $H'$ in \eqref{E:b_map}; equivalently each boundary hypersurface of
$M$ is mapped surjectively either onto some boundary hypersurface of $N$ or
onto $N$ itself.

\begin{thm}
Every admissible map $f : X \to Y$ extends to a unique b-map $\wh f: \wh X \to \wh Y$. Moreover,
if $f$ is an admissible quotient, then $\wh f$ is a b-fibration.
\label{T:mb_functor}
\end{thm}

\begin{proof}
We consider first the case that $f : X \to X/K$ is an admissible quotient, with
$\cS_X = \set{\set 0, K, X}$, so we must show that $\wh X = [\ol X; \pa \ol K]
\to \ol{X/K}$ is a b-fibration.
For this we choose a complement $X = W\oplus K$ and
write $x = (x_1,x_2) \in W\oplus K\cong\bbR^{n-k}\oplus\bbR^k$ in ``product radial'' coordinates
\[
\begin{gathered}
	x = (x_1,x_2) = R\omega =  R (r \xi_1,  s\xi_2),
	\\ R= \abs x,
	\quad \omega = \frac{x}{R}
	\quad r = \abs{ \frac{x_1}{R}}, \quad s = \abs{\frac{x_2}{R}}= \sqrt{1- r^2},
	\quad \xi_1 = \frac{x_1}{Rr}\in \bbS^{n-k-1}, \quad \xi_2 = \frac{x_2}{Rs} \in \bbS^{k-1}.
\end{gathered}
\]
As with standard polar coordinates if $R$ or $r$ or $s$ vanishes than the
spherical variables are under-determined; it is more accurate to view the
coordinates as a map
\[
	\bbR_+\times [0,1]\times \bbS^{n-k-1}\times \bbS^{k-1} = \set{(R,r,\xi_1,\xi_2)} \mapsto (Rr\xi_1,R(\sqrt{1-r^2})\xi_2) \in X
\]
which is a diffeomorphism away from the zero sets of $R$, $r$ or $s = \sqrt{1-r^2}$.
In any case, coordinates on the radial compactification of $X$ are given by
$(\rho, \omega) = (\rho,(r,\xi_1),(s,\xi_2))$, where $\rho = 1/R$, and coordinates on $\ol{X/K}$
are given by $(\sigma, \xi_1)$ where $\sigma = 1/Rr$.
The submanifold $\pa \ol K$ is given by $\set{\rho = r = 0}$, and its blow-up in $\ol X$
is parameterized near the corner by coordinates $(\sigma, (r, \xi_1), (\sqrt{1 - r^2}, \xi_2))$,
where again $\sigma = \rho/r = 1/Rr$, and $\set{r = 0}$ is no longer singular.
The projection map $X \to X/K$, $(x_1,x_2) \mapsto x_1$ extends by continuity
to the map $[\ol X; \pa \ol K] \to \ol{X/K}$, $(\sigma, r,\xi_1, \xi_2) \mapsto (\sigma, \xi_1)$,
which is manifestly a b-fibration.
%
%

Returning to the general case of an admissible quotient, let us assume
inductively that we have a b-fibration $[\ol X; \pa \cS'] \to [\ol{X/K}; \pa
(\cS'/K)]$ for an intersection closed subset $\cS' \subset \cS_X$ of the linear system in
$X$, the base case $\cS' = \set{\set 0, K, X}$ having been shown above.
Let $W \in \cS_X \setminus \cS'$ be a minimal element, meaning there is no
$V \in \cS_X \setminus \cS'$ with $V \subset W$. In particular $W \cap V \in \cS'$ for all $V \in \cS'$.
There are three possibilities:
\begin{enumerate}
[{\normalfont (1)}]
\item $W$ is contained in $K$. In this case the image of $W$ in $X/K$ is the trivial subspace, and does
not induce an additional blow-up in the target $[\ol{X/K}; \pa (\cS'/K)]$. In
the domain, the composition of the blow-down $[\ol X; \pa \cS', \pa \ol W] \to
[\ol X; \pa \cS']$ with the b-fibration to $[\ol{X/K}; \pa (\cS'/K)]$ is again
a b-fibration, since the front face associated to $\pa \ol W$ maps into the
interior of the target.
\item $W$ intersects $K$ transversally. In this case the image of $W$ in $X/K$ is the whole space, and again
does not induce an additional blow-up in the target. In the blow-up $[\ol X;
\pa \cS', \pa \ol W]$, the front face maps to the original radial boundary of
$[\ol{X/K}; \pa(\cS'/K)]$, which is of codimension one, so this is again a
b-fibration.
\item If neither of the above holds, then $W$ descends to the proper nontrivial subspace $W/(W\cap K)$ in $X/K$.
If we haven't yet blown up (the lift of) $\pa \ol{W/(W\cap K)}$ in the target,
then we may blow this up along with its preimage in the domain, which is (the lift
of) $\pa \ol {W + K}$; by admissibility $W + K$ is an element of $\cS_X$.  The
old b-fibration lifts to these blow-ups, and is again a b-fibration since the
new front face of the domain is mapped onto the new front face of the target
which has codimension one.
In so doing we may assume that $W + K \in \cS'$.  Then if $W \neq W + K$, the
composition of the blow-down $[\ol X; \pa \cS', \pa \ol W] \to [\ol X; \pa
\cS']$ with the b-fibration to $[\ol {X/K}; \pa (\cS'/K)]$ is a b-fibration
since the front face maps onto the hypersurface associated to $\pa\ol{W/(W\cap K)}$.
\end{enumerate}

In any case, by Theorem~\ref{T:comm_blowup_order}, the space $[\ol X; \pa
\cS', \pa \ol W]$ is diffeomorphic to a size order blow-up $[\ol X; \pa \cS'']$
where $\cS'' = \cS' \cup \set{W}$, and we may then replace $\cS'$ by $\cS''$
to complete the induction.
This completes the proof that admissible quotients lift to b-fibrations.

For a general admissible map $f : X \to Y$, we make a series of reductions.
By admissibility, $f^{-1}(S_Y)$ is an intersection closed linear subsystem of
$\cS_X$.
Then $[\ol X; \pa \cS_X]$ admits a b-map to $[\ol X; \pa f^{-1}(\cS_Y)]$ by
Theorem~\ref{T:comm_blowup_order}, so we can suppose from now on that $\cS_X
= f^{-1}(\cS_Y)$.
We may factor $f$ as the quotient $X \to X/K$ and an injection $X/K
\hookrightarrow Y$, where $K = \ker f$.
Since every element of $f^{-1}(\cS_Y)$ contains $K := \ker f$, the map $X \to
X/K$ is an admissible quotient, which extends to a b-map as shown above, so it
remains to consider the case that $X$ is a subspace of $Y$, with $\cS_X = X \cap \cS_Y$.
For each $W \in \cS_Y$ such that $W \cap X$ is a proper subspace of $X$, it follows
that the lift of $\ol X$ to the blow-up of $\pa \ol W$ in $\ol Y$ is diffeomorphic
to $[\ol X; \pa \ol {W \cap X}]$. Indeed, this is a general property of blow-up for cleanly
intersecting submanifolds which is elementary to check in local coordinates.
On the other hand, if $W \supset X$, then the lift of $\ol X$ to the blow-up of $\pa \ol W$
in $\ol Y$ is diffeomorphic again to $\ol X$.
It follows iteratively then that $\ol X \subset \ol Y$ lifts to a p-submanifold of $\wh Y$
which is diffeomorphic to $\wh X = [\ol X; \pa \cS_Y \cap X]$.
\end{proof}

\section{Boundary faces} \label{S:boundary}
A linear system $\cS_X$ is a set which is partially ordered by
inclusion, has minimal and maximal elements, and for each pair
$V$, $W$ of elements has a unique infimum
$V \cap W$.
It is notationally convenient at this point to use $\cS_X$ as an abstract partially ordered
indexing set, and from now on we will use Greek letters $\lambda,\mu \in \cS_X$ for
elements, with the order and infimum denoted by $\lambda \leq \mu$ and $\lambda
\wedge \mu$, respectively.
We denote the minimal element by $0$ and sometimes denote the maximal
element by $1$.
We write $X_\lambda$ instead of $\lambda$ when we wish to emphasize the actual subspaces of
$X$, thus $X_0 = \set 0 $, $X_1 = X$, and $X_{\lambda \wedge \mu} = X_\lambda
\cap X_\mu$.

\begin{thm}[c.f.\ \cite{vasy2001propagation}]
There is a bijective correspondence between
boundary hypersurfaces of $\wh X$
and
$\cS_X \setminus \set 0$,
under which $\lambda \in \cS_X \setminus \set 0$
corresponds to a hypersurface $N_\lambda$ diffeomorphic to the product
\begin{equation}
	N_\lambda \cong B_\lambda \times F_\lambda,
	\quad B_\lambda = [\pa \ol X_\lambda; \set{\pa \ol X_\mu : \mu < \lambda}],
	\quad F_\lambda = \wh{X/X_\lambda}.
	\label{E:bdy_hyp}
\end{equation}
Moreover $N_{\lambda_1} \cap \cdots \cap N_{\lambda_k} \neq \emptyset$ if and
only if $\set{\lambda_1 <\cdots <\lambda_k} \subset \cS_X$ is a totally ordered
subset.
In this case
\[
\begin{gathered}
	N_{\lambda_1}\cap \cdots \cap N_{\lambda_k} \cong B_{\lambda_1}\times B_{\lambda_1,\lambda_2} \times \cdots \times B_{\lambda_{k-1},\lambda_k} \times F_{\lambda_k},
	\\B_{\nu,\mu} = [\pa \ol{X_{\mu}/X_{\nu}}; \{\pa \ol {X_{\kappa}/(X_{\kappa} \cap X_{\nu})} : \kappa < \mu\}].
\end{gathered}
\]
\label{T:boundary_faces}
\end{thm}
In particular there is a maximal $\cS_X$ boundary hypersurface $N_1 = B_1$ with
respect to $(\cS_X,\leq)$ which in the many body literature is sometimes called
the {\em free region}.
$F_\lambda = \wh{X/X_\lambda}$ is again a many-body compactification with
system $\cS_{F_\lambda} \cong \set{\mu \in \cS_X : \mu \leq \lambda}$ and
maximal element $1 = \lambda$.
The other factor $B_\lambda$, which may be identified with the free region of
$\wh X_\lambda$, is not a many-body compactification, but it is a manifold with
corners whose boundary hypersurfaces are indexed by the ordered set $\set{ \mu
\in \cS_X : \lambda > \mu}$.

It is convenient to set $N_0 = \wh X$, which is not of course a boundary
hypersurface, but is consistent with \eqref{E:bdy_hyp}, since $B_0$ is a point
(being the radial compactification of $\set{0}$), and $F_0 = \wh {X/\set 0} =
\wh X$.

\begin{proof}[Proof of Theorem~\ref{T:boundary_faces}]
The definition of $\wh X$ makes it clear that there is indeed a boundary
hypersurface $N_\lambda$ for each element $\lambda$ of $\cS_X \setminus \set
0$.
To see the structure of $N_\lambda$, consider its origin as the submanifold
$\pa \ol X_\lambda$ inside $\ol X$.
We first blow-up in $\ol X$ all those submanifolds $\pa \ol X_\mu$ such that
$\mu < \lambda$ (note that if $\mu$ precedes $\lambda$ in the size order but
not in the original partial order on $\cS_X$ then $\pa \ol X_\mu$ and $\pa \ol
X_\lambda$ do not meet), after which the lift of $\pa \ol {X_\lambda}$ is
diffeomorphic to the space $B_\lambda = [\pa \ol X_\lambda; \set{\pa \ol X_\mu
: \mu < \lambda}]$.

We then blow-up this lift $B_\lambda$ of $\pa \ol X_\lambda$ itself,
introducing as a front face the inward pointing spherical normal bundle of
$B_\lambda$ in $\ol X$.
However, the inward spherical normal bundle of $\pa \ol X_\lambda$ in $\ol X$
is equivalent (essentially by definition) to the radial compactification of the
normal bundle to $\ol X_\lambda$ in $\ol{X}$ and this remains true even after
passing to $B_\lambda$ by blow-up.
Since the spaces are linear, this bundle is simply the product $B_\lambda
\times \ol{X/X_\lambda}$.

Finally, we proceed to blow-up the lifts of those $\pa \ol X_\mu$ such that
$\mu > \lambda$, which meet $B_\lambda \times
\ol{X/X_\lambda}$ in the submanifolds $B_\lambda \times \pa
\ol{X/X_\mu}$ sitting inside the boundary face $B_\lambda \times \pa
\ol{X/X_\lambda}$, from which the first claim follows.

The second claim follows from the fact that for any pair of subspaces
$X_\lambda$ and $X_\mu$ such that $X_\lambda \not \subset X_\mu$ and $X_\mu
\not \subset X_\lambda$, the lifts of $\pa \ol X_\lambda$ and $\pa \ol X_\mu$
are made disjoint by the blow-up of $\pa \ol{X_{\lambda \wedge \mu}}$.
If, on the other hand, $\lambda < \mu$, then $N_\mu$ and $N_\lambda$ meet
precisely in the boundary hypersurface of $N_\lambda = B_\lambda \times
F_\lambda$ introduced by the blow up of $B_\lambda \times \pa
\ol {X_\mu/X_\lambda}$ inside $B_\lambda \times \pa \ol{X/X_\lambda}$, which
corresponds in the second factor to the face $N_{\lambda,\mu} = B_{\lambda,\mu}
\times F_\mu$ of $F_\lambda$, giving $N_\mu \cap N_\lambda = B_\lambda \times
B_{\lambda,\mu}\times F_\mu$.
The general case follows by induction.
\end{proof}

\begin{rmk}\label{Rmk:qfb}
The structure \eqref{E:bdy_hyp} is a primary example of what is known variously
in the manifolds with corners literature as
a {\em resolution structure} \cite{albin2010resolution},
an {\em iterated fibration structure} \cite{albin2012signature},
or (the term we use here)
a {\em fibered corners structure} \cite{conlon2016quasi}.
In general, this means a manifold with corners $M$ whose boundary hypersurfaces
$N_\lambda$ are indexed by a partially ordered set
and are equipped with fibrations $\phi_\lambda : N_\lambda \to B_\lambda$ with
typical fiber $F_\lambda$, such that:
\begin{enumerate}
[{\normalfont (i)}]
\item Each $F_\lambda$ and $B_\lambda$ are also manifolds with corners.
\item $N_{\lambda_1} \cap \cdots \cap N_{\lambda_N} \neq \emptyset$ if and only if $\lambda_1 < \cdots < \lambda_N$
is a totally ordered chain.
\item If $\lambda < \mu$, then $\phi_{\lambda} \rst_{N_\lambda \cap N_\mu} : N_\lambda \cap N_\mu \to B_\lambda$
is a fibration whose typical fiber is a boundary hypersurface $\pa_\mu F_\lambda$ of
$F_\lambda$, while $\phi_\mu \rst_{N_\lambda \cap N_\mu}$ is a restriction of $\phi_\mu$
over a boundary hypersurface $\pa_\lambda B_\mu := \phi_\mu (N_\lambda \cap N_\mu)$ of
$B_\mu$; moreover there is a fibration $\phi_{\lambda,\mu} : \pa_\lambda B_\mu \to B_\lambda$
such that $\phi_{\lambda,\mu} \circ \phi_\mu = \phi_\lambda$.
\item Every boundary hypersurface of $F_\lambda$ is of the form $\pa_\mu F_\lambda$ for some $\mu > \lambda$
and likewise every boundary hypersurface of $B_\mu$ is of the form $\pa_\lambda B_\mu$
for some $\lambda < \mu$.
In particular it follows that each $F_\lambda$ and $B_\lambda$ has a fibered
corners structure induced by the maps $\phi_\lambda$ and $\phi_{\mu,\lambda}$,
respectively.
\end{enumerate}
(The fibered corners structure on a many body compactification $\wh X$ is
suggested by \eqref{E:bdy_hyp}, with $F_\lambda$ the fiber and
$B_\lambda$ the base, which is the correct interpretation from the point of
view of QFB metrics (see below).
However, since the fibrations are products in this case, $\wh X$ also admits an
inequivalent fibered corners structure with the indexing set, as well as the
roles of the $F_\lambda$ and $B_\lambda$, reversed!)

In fact, by \cite{albin2012signature}, a manifold $M$ with fibered corners is
equivalent to the resolution of a {\em smoothly stratified space} $\wt M =
M/{\sim}$, obtained by collapsing the fibers of each boundary hypersurface, i.e.,
taking the quotient by the equivalence relation where $p \sim q$ if
$\phi_\lambda(p) = \phi_\lambda(q)$ for some $\lambda$.
Conversely, a smoothly stratified space may be defined intrinsically as a
stratified space $S$ with control data in the sense of Mather---in particular,
the strata admit tubular neighborhoods in $S$ which are assumed to be locally
trivial cone bundles (see \cite{albin2012signature} for a detailed
definition)---and then the resolution by iterative radial blowup of the strata
of $S$ yields a manifold with fibered corners.

In our case, the smoothly stratified space in question is simply the original radial
compactification $\ol X$, with strata consisting of the interior of $\ol X$
along with the boundaries $\pa \ol X_{\lambda} \setminus \set{\pa \ol X_\mu :
\mu < \lambda}$ of the subspaces in the system $\cS_X$.

In addition to this combinatorial topological structure, there is a natural
geometric structure on $\wh X$ induced by any Euclidean metric on $X$.
Such structure may characterized equivalently in terms of the vector fields
which are bounded with respect to such a metric, and in turn these vector
fields admit a metric-independent description.
In the general setting of a manifold $M$ with fibered corners, a {\em
quasi-fibered boundary} (QFB) structure (see \cite{conlon2016quasi}) is defined by a Lie subalgebra
$\cV_\QFB(M) \subset \cV_\b(M)$ of vector fields (here $\cV_\b(M)$ is the
algebra of vector fields tangent to all boundary faces of $M$), defined as
those vector fields $V$ such that
\begin{enumerate}
[{\normalfont (i)}]
\item
$V$ is tangent to the fibers $F_\lambda$ at
each boundary hypersurface $N_\lambda$, and
\item
$V(\rho) \in
\rho^2C^\infty(M)$ where $\rho = \prod_\lambda \rho_\lambda$ is a choice of
total boundary defining function.
\end{enumerate}
A {\em QFB metric} may be then be defined as a Riemannian metric on the
interior such that the pointwise norm of each $V \in \cV_\QFB(M)$ extends
smoothly up to the boundary of $M$.
In the special case that the boundary fibrations are trivial for maximal $\lambda$
(so $N_\lambda \cong B_\lambda$ with $\phi_\lambda \cong \id$), a QFB structure is
known as a {\em quasi-asymptotically conic} (QAC) structure (c.f. \cite{degeratu2014fredholm,conlon2016quasi}).

That the lift of a Euclidean metric on $X$ furnishes a QAC metric on $\wh X$ is
trivial to verify: indeed, on the radial compactification $\ol X$, the inverse
radial function $\rho = r^\inv$ furnishes a canonical boundary defining
function, which then lifts to a total boundary defining function on $\wh X$.
Moreover, the vector fields which are bounded with respect to the Euclidean
metric are precisely those $V$ on $\ol X$ such that $V\rho \in \rho^2C^\infty(\ol
X)$ (these are the {\em scattering vector fields} in the sense of Melrose \cite{melrose1994spectral}),
and these are easily seen to lift to be tangent to the boundary fibrations on
$\wh X$.
\end{rmk}

From now on we consider the b-fibration $\wh f : \wh X \to \wh Y$ associated to
an admissible quotient $f : X \to Y$.
As $\wh f$ is a b-fibration, the smallest face of $\wh Y$ containing the image
$f(N_\lambda)$ of each hypersurface $N_\lambda$ is either a hypersurface
$M_\mu$ of $\wh Y$ or $\wh Y$ itself.
We recall that the restriction of a b-fibration to an arbitrary boundary face
of the domain is again a b-fibration.

Fix $\lambda \in \cS_X$ and let $\mu = f(\lambda) \in \cS_Y$.
Note then that
\[
	f_\lambda := f \rst_{X_\lambda} : X_\lambda \to Y_\mu,
	\quad \text{and}\quad
	f/f_{\lambda} : X/X_\lambda \to Y/Y_\mu
\]
are admissible linear maps, the former of which sends the
maximal element of $\cS_{X_\lambda}$ to the maximal element in $\cS_{Y_\mu}$ and the
latter of which is an admissible quotient.

\begin{thm}
Let $f : X \to Y$ be an admissible quotient,
let $N_\lambda$ be a boundary hypersurface of $\wh X$, and set $\mu =
f(\lambda) \in \cS_Y$.
Then
the restriction of $\wh f$
to $N_\lambda$ is a b-fibration onto $M_\mu$ which is diffeomorphic to the product map
\[
	\wh f_\lambda \times \wh {f/f_\lambda} :
	B_\lambda \times F_\lambda 
	\to B_\mu \times F_\mu 
\]
where $B_\lambda \subset \wh X_\lambda$ and $B_\mu \subset \wh Y_\mu$ are the
free regions of $\wh {X_\lambda}$ and $\wh {Y_\mu}$, respectively, and $F_\lambda = \wh{X/X_\lambda}$
and $F_\mu = \wh{Y/Y_\mu}$ as above.
\label{T:mb_map_on_hs}
\end{thm}

Note that the statement applies as well in the case that $f(\lambda) = 0$
(i.e., $X_\lambda \in \ker f$), in which case $N_\lambda$ maps onto $M_0 = \wh Y$
itself, via the map
\[
	0 \times \wh {f/f_\lambda} : B_\lambda \times \wh{X/X_\lambda} \to \set 0 \times \wh Y.
\]

\begin{proof}[Proof of Theorem~\ref{T:mb_map_on_hs}]
By uniqueness of the continuous extensions of $f$, $f_\lambda$ and
$f/f_\lambda$ to the compactifications of their respective domains, it suffices
to show that $\wh f$ and $\wh f_\lambda \times \wh{f/f_\lambda}$ agree on the
interior of $N_\lambda$.
As noted in the proof of Theorem~\ref{T:boundary_faces} above, this
interior may be identified with the normal bundle of (an open dense subset of)
$\pa \ol {X_\lambda}$, which is just the product $\pa \ol {X_\lambda} \times X/X_\lambda$.
On the other hand, as a linear map $f$ may be identified with its normal differential
along $X_\lambda$, which may be in turn identified with the product map
$f_\lambda \times f/f_\lambda$ from $X \cong X_\lambda \times X/X_\lambda$ to $Y \cong Y_\mu \times Y/Y_\mu$.
The extension of this by continuity over $\pa \ol X_{\lambda} \times X/X_\lambda$ agrees
by definition with $\wh f \times f/f_\lambda$, and therefore with $\wh f \times \wh{f/f_\lambda}$
on the interior of its domain.
\end{proof}

\section{An application} \label{S:appl_scatprod}
In \cite{melrose2008scattering}, Melrose and Singer consider the problem of compactifying
the products $X^n$ of a vector space $X$ as manifolds with corners $X^n_{\mathrm{sc}}$ in such a way that
\begin{enumerate}
[{\normalfont (i)}]
\item
$X^1_{\mathrm{sc}} = \ol X$, the radial compactification,
\item
the action of the permutation group $\Sigma_n$ lifts to $X^n_{\mathrm{sc}}$,
\item
the various projections
\begin{equation}
	\pi_{I} : X^n \ni (u_1,\ldots,u_n) \mapsto (u_{i_1},u_{i_2},\ldots,u_{i_k}) \in X^k,
	\quad I = \set{1 \leq i_1< i_2 < \cdots <i_k \leq n}
	\label{E:projections}
\end{equation}
lift to b-fibrations $X^n_{\mathrm{sc}} \to X^k_{\mathrm{sc}}$, and
\item \label{I:difference_maps}
the difference maps
\begin{equation}
	\delta_{ij} : X^n \ni (u_1,\ldots,u_n) \mapsto u_i - u_j \in X,
	\quad i \neq j
	\label{E:difference_maps}
\end{equation}
lift to b-fibrations $ X^n_{\mathrm{sc}} \to \ol X$.
\end{enumerate}
In fact they work in the setting of a general compact manifold with boundary
$M$ in place of $\ol X$, so generalizing to higher $n$ the {\em scattering spaces}
$M^2_{\mathrm{sc}}$ and $M^3_{\mathrm{sc}}$ introduced in
\cite{melrose1994spectral} to support kernels of pseudodifferential operators
and their compositions.
(Note that \eqref{I:difference_maps} does not make sense at this level of generality.)
In order to work in this general setting, Melrose and Singer must
start with the manifolds with corners $M^n = (\ol X)^n$ and develop quite a few
delicate and technical results about commutativity of blow-up of various
families of submanifolds in order to obtain spaces
satisfying the required properties.

On the other hand, provided one is willing to stick to the original setting of
vector spaces, the comparatively simpler theory developed here
furnishes an immediate solution.
Indeed, within the product $X^n$ consider two families of subspaces: the {\em
axes} $\set{(u_1,\ldots,u_n) \in X^n : u_i = 0 \text{ for $i \in J$}}$ and the
{\em diagonals} $\set{(u_1,\ldots,u_n) : u_i = u_j \text{ for $i,j \in J$}}$,
where here $J$ runs over all subsets of $\set{1,\ldots,n}$.
The following is immediate.
\begin{thm}
Let $X$ be a vector space, and for $n \in \bbN$, equip $X^n$ with the linear
system generated by all axes and diagonals. Then the permutations $\Sigma_n \ni
\sigma : X^n \to X^n$, the projections \eqref{E:projections}, and the
difference maps \eqref{E:difference_maps} are all admissible quotients, hence
lift to respective b-fibrations
\[
	\wh \sigma : \wh{X^n} \to \wh{X^n},
	\quad \wh \pi_I : \wh{X^n} \to \wh{X^k},
	\quad \text{and} \quad \wh \delta_{ij} : \wh{X^n} \to \ol X.
\]
\label{T:mb_soln}
\end{thm}

\begin{rmk}
The main reason why the ``many-body space'' solution to the above compactification
problem is so much simpler than the ``scattering products'' solution is the availability
in the linear setting of the radial compactification $\ol{X^n}$ of the product
as an alternative to the product $(\ol X)^n$ of the radial compactifications.
If $M$ is a manifold with boundary, it is possible to define the analogue of
the radial compactification of the products $(M^\circ)^n$, though unless $\pa
M$ is a sphere these will be singular stratified spaces, and it is far from
clear that an analogue of Theorem~\ref{T:comm_blowup_order} holds in such a category.

Finally, we note here that the spaces $\wh{X^n}$ and $X^n_{\mathrm{sc}}$ are {\em not}
diffeomorphic if $n \geq 3$, the verification of which we leave as an exercise
to the interested reader.
\end{rmk}

\bibliographystyle{amsalpha}
\bibliography{references}
\end{document}